\documentclass[12pt]{amsart}
\usepackage{amscd,amssymb,amsmath}
\usepackage[all]{xy}
\usepackage{amsfonts}
\usepackage{amssymb}
\usepackage{epic}
\usepackage{color}

\newtheorem{theorem}{Theorem}[section]
\newtheorem{lemma}[theorem]{Lemma}
\theoremstyle{definition}

\newtheorem{corollary}[theorem]{Corollary}

\newtheorem{conjecture}[theorem]{Conjecture}
\theoremstyle{remark}
\newtheorem{remark}{Remark}[section]
\newtheorem{teo}{Theorem}[section]

\theoremstyle{definition}
\newtheorem{dfn}[teo]{Definition}
\newtheorem{rk}[teo]{Remark}
\newtheorem{ex}[teo]{Example}
\newtheorem{que}[teo]{Question}
\numberwithin{equation}{section}

\newcommand{\N}{\mathcal N}
\newcommand{\R}{\mathcal R}

\newcommand{\bea} {\begin{eqnarray*}}
\newcommand{\beq} {\begin{equation}}
\newcommand{\bey} {\begin{eqnarray}}
\newcommand{\eea} {\end{eqnarray*}}
\newcommand{\eeq} {\end{equation}}
\newcommand{\eey} {\end{eqnarray}}

\def\Z{{\mathbb Z}}

\bibliography{ref}
\def\Fix{\operatorname{Fix}}

\def\N{{\mathbb N}}

\def\<{\langle}
\def\>{\rangle}

\def\R{{\mathbb R}}
\def\Z{{\mathbb Z}}

\def\M{{\mathcal M}}

\def\tr{\mathop{\rm tr}\nolimits}

\def\Ind{\operatorname{Ind}}
\def\ind{\operatorname{ind}}

\def\1{\mathbf 1}

\newcommand{\be}{\begin{enumerate}}
\newcommand{\ee}{\end{enumerate}}
\newcommand{\bq}{\begin{question}}
\newcommand{\eq}{\end{question}}
\newcommand{\bcj}{\begin{conjecture}}
\newcommand{\ecj}{\end{conjecture}}
\newcommand{\bc}{\begin{corollary}}
\newcommand{\ec}{\end{corollary}}
\newcommand{\bl}{\begin{lemma}}
\newcommand{\el}{\end{lemma}}
\newcommand{\btl}{\begin{technicalLemma}}
\newcommand{\etl}{\end{technicalLemma}}
\newcommand{\bp}{\begin{proposition}}
\newcommand{\ep}{\end{proposition}}
\newcommand{\bft}{\begin{fact}}
\newcommand{\eft}{\end{fact}}
\newcommand{\brk}{\begin{remark}}
\newcommand{\erk}{\end{remark}}
\newcommand{\bd}{\begin{Dn}}
\newcommand{\ed}{\end{Dn}}

\def\M{{\mathcal M}}

\def\tr{\mathop{\rm tr}\nolimits}

\numberwithin{equation}{section}
\newcommand{\Q}{\mathbf Q}

\newcommand{\fix}[1]{\mathop{\textrm{Fix}}(#1)}

\def \crit{\operatorname{crit}}

\def\<{\langle}
\def\>{\rangle}

\newcommand{\de}{\mathrm{d}}

\newcommand{\De}{\mathrm{D}}

\newcommand{\Cc}{\mathcal{C}}

\renewcommand{\to}{\rightarrow}
\newcommand{\To}{\longrightarrow}
\newcommand{\Mapsto}{\longmapsto}

\newcommand{\inclusion}{\hookrightarrow}
\newcommand{\p}{\partial}

\newcommand{\Of}{\Omega_{\phi}}
\newcommand{\ophi}{\omega_{\phi}}
\DeclareMathOperator{\id}{id}

\DeclareMathOperator{\im}{image}

\DeclareMathOperator{\area}{area}

\DeclareMathOperator{\symp}{Symp}
\DeclareMathOperator{\diff}{Diff}

\DeclareMathOperator{\sign}{sign}
\DeclareMathOperator{\flux}{Flux}
\DeclareMathOperator{\grow}{Growth}

\def\tr{\mathop{\rm tr}\nolimits}

\def\Ind{\operatorname{Ind}}
\def\ind{\operatorname{ind}}

\def \bd {\partial}

\newcommand{\HF}{HF_*}
\newcommand{\CF}{CF_*}

\def\Fix{\operatorname{Fix}}

\def\N{{\mathbb N}}

\begin{document}

\title[ Floer homology and the Poincare - Birkhoff theorem]
{Nielsen  theory, Floer homology and a generalisation of
 the Poincare - Birkhoff theorem}

\author{Alexander Fel'shtyn}
\address{ Institute of Mathematics, University of Szczecin,
ul. Wielkopolska 15, 70-451 Szczecin, Poland
and Department of Mathematics, Boise State University, 1910
University Drive, Boise, Idaho, 83725-155, USA }
\email{felshtyn@diamond.boisestate.edu, felshtyn@mpim-bonn.mpg.de}

\keywords{Nielsen number, 
symplectic Floer homology, Poincare - Birkhoff theorem, Arnold conjecture }
\subjclass{ 37C25;  53D; 37C30; 55M20}

\begin{abstract}
The purpose of this mostly expository paper is to discuss  a connection between Nielsen fixed point theory and  symplectic  Floer homology for symplectomorphisms of surface  and  a  calculation  of  Seidel's symplectic Floer homology for different mapping classes. We also describe  symplectic zeta functions and an  asymptotic
symplectic  invariant. A generalisation of the  Poincare - Birkhoff fixed point theorem  and Arnold conjecture   is proposed.

\end{abstract}

\maketitle

\tableofcontents

\section{Introduction}

Before discussing the main   results of the paper, we briefly describe the few
basic notions of Nielsen fixed point theory which will be used.
We assume  $X$ to be a connected, compact
polyhedron and $f:X\rightarrow X$ to be a continuous map.
Let $p:\tilde{X}\rightarrow X$ be the universal cover of $X$
and $\tilde{f}:\tilde{X}\rightarrow \tilde{X}$ a lifting
of $f$, i.e. $p\circ\tilde{f}=f\circ p$.
Two liftings $\tilde{f}$ and $\tilde{f}^\prime$ are called
{\sl conjugate} if there is a $\gamma\in\Gamma\cong\pi_1(X)$
such that $\tilde{f}^\prime = \gamma\circ\tilde{f}\circ\gamma^{-1}$.
The subset $p(Fix(\tilde{f}))\subset Fix(f)$ is called
{\sl the fixed point class of $f$ determined by the lifting class $[\tilde{f}]$}.Two fixed points $x_0$ and $x_1$ of $f$ belong to the same fixed point class iff  there is a path $c$ from $x_0$ to $x_1$ such that $c \cong f\circ c $ (homotopy relative endpoints). This fact can be considered as an equivalent definition of a non-empty fixed point class.
 Every map $f$  has only finitely many non-empty fixed point classes, each a compact  subset of $X$.
A fixed point class is called {\sl essential} if its index is nonzero.
The number of essential fixed point classes is called the {\sl Nielsen number}
of $f$, denoted by $N(f)$.The Nielsen number is always finite.
 $N(f)$ is a  homotopy type  invariant.
In the category of compact, connected polyhedra, the Nielsen number
of a map is, apart from certain exceptional cases,
 equal to the least number of fixed points
 of maps with the same homotopy type as $f$.

 Recently a  connection between symplectic Floer homology and Nielsen fixed
point theory was discovered \cite{ff, G}. The author came to the  idea that Nielsen fixed point theory  is  connected  with  symplectic  Floer homology
theory  of symplectomorphisms of  surfaces   at the Autumn 2000, after conversations with Joel Robbin and Dan Burghelea.
This connection is based on   the following simple fact:  Gromov pseudoholomorphic curve in symplectic fixed point  theory
 is nothing else as a Nielsen holomorphic disc.
In the  dimension two a diffeomorphism is symplectic if it preserves
area. As  a consequence, the symplectic geometry of surfaces lacks many of the interesting phenomena which are encountered in higher dimensions. For  example, two symplectic automorphisms of a closed surface are symplectically isotopic iff they are homotopic, by a theorem of Moser \cite{Mo}.
On other hand symplectic fixed point theory is very nontrivial in dimension 2.
A selebrated landmark in this subject is Poincar\'e's last geometric theorem, proved by
Birkhoff \cite{B} in the 1925, which asserts that an area-preserving twist map of the annulus must have
at least two distinct fixed points.
It is known that symplectic Floer homology for  surface symplectomorphism  is the $d=1$ part of periodic Floer homology
of this  symplectomorphism \cite{HS}, a Floer homology theory whose chain complex is generated by certain
multisets of periodic orbits and whose differentials count certain embedded pseudoholomorphic curves.
This theory is conjectured \cite{HS} to be isomorphic to the Seiberg - Witten Floer homology \cite{KM}
 of the  mapping torus of symplectomorphism in certain  $spin^c$-structures and also  is conjectured to agree with Heegaard Floer homology $HF^+$ \cite{OS} of the mapping torus.
It is known also that symplectic Floer homology of symplectomorphism of surface  is a simple model for the instanton Floer homology of the
  mapping torus of this symplectomorphism \cite{ds}.

In the chapter \ref{sec:floer} we discuss  the connection between  symplectic Floer homology theory  and Nielsen fixed point theory and  the calculations  of 
Seidel's symplectic Floer homology for  periodic  \cite{G, ff} and
algebraically finite mapping classes \cite{G}.
In the chapter \ref{sec:zeta} we  describe symplectic zeta functions and an  asymptotic invariant of monotone symplectomorphism. 
In the chapter \ref{sec:problem}  we discuss a generalisation of the  Poincar\'e - Birkhoff theorem and Arnold conjecture  and  the calculations  of 
Seidel's symplectic Floer homology for pseudo-Anosov and reducible mapping classes \cite{cot}
The results of this paper were announced on the  Symplectic Geometry Seminars
at Stanford and Princeton    in April and October, 2007.

The author is  grateful to
 A. Cotton-Clay,  Ya. Eliashberg, Wu-Chung Hsiang, E. Ionel,  J. Kedra,   K. Ono,   Y. Rudyak, G. Tian, V. Turaev and M. Usher    for stimulating discussions and comments.

Parts of this article were written while the author  was visiting the
Max-Planck-Institute f\"ur Mathematik,
Bonn in 2004-2005.
The author would like to thank the Max-Planck-Institute f\"ur Mathematik,
Bonn for kind hospitality and support.

\section{Nielsen fixed point theory and symplectic Floer homology} \label{sec:floer}

\subsection{ Symplectic Floer homology}

\subsubsection{Monotonicity}

In this section we discuss the notion of monotonicity as defined in \cite{S,G}.
Monotonicity plays important role for  Floer homology in two
dimensions.
Throughout this article, $M$ denotes a closed connected and oriented
2-manifold of genus $\geq2$.   Pick an everywhere positive two-form $\omega$
on $M$.

Let $\phi\in\symp(M,\omega)$, the group of  symplectic automorphisms
 of the two-dimensional symplectic
manifold $(M,\omega)$.
The mapping torus of $\phi$,  $T_\phi = \R\times M/(t+1,x)\sim(t,\phi(x)),$
is a 3-manifold fibered over $S^1=\R/\Z$.
There are two natural second
cohomology classes on $T_\phi$, denoted by $[\ophi]$ and $c_\phi$. The first one is
represented by the closed two-form $\ophi$ which is induced from the pullback
of $\omega$ to $\R\times M$. The second is the Euler class of the
vector bundle
$
V_\phi =
\R\times T M/(t+1,\xi_x)\sim(t,\de\phi_x\xi_x),
$
which is of rank 2 and inherits an orientation from $TM$.

$\phi\in\symp(M,\omega)$ is called {\bf monotone}, if
$
[\omega_\phi] = (\area_\omega(M)/\chi(M))\cdot c_\phi
$
in $H^2(T_\phi;\R)$; throughout this article
$\symp^m(M,\omega)$ denotes the set of
monotone symplectomorphisms.

Now $H^2(T_\phi;\R)$ fits into the following short exact sequence \cite{S,G}
\begin{equation}\label{eq:cohomology}
0 \To  \frac{H^1(M;\R)}{\im(\id-\phi^*)}
\stackrel{d}{\To} H^2(T_\phi;\R)
\stackrel{r^*}{\To} H^2(M;\R),
\To 0.
\end{equation}
where the map $r^*$  is restriction to the fiber.
The map $d$ is defined as follows.
Let $\rho:I\to\R$ be a smooth function which vanishes near $0$ and $1$ and
satisfies $\int_0^1\!\rho\,\de t=1$.
If $\theta$ is a closed 1-form on $M$, then
$\rho\cdot\theta\wedge\de t$ defines a closed 2-form on $T_\phi$; indeed
$
d[\theta] = [\rho\cdot\theta\wedge\de t].
$
The map $r:M\inclusion T_\phi$ assigns to each $x\in M$ the
equivalence class of $(1/2,x)$.
Note, that $r^*\ophi=\omega$ and $r^*c_\phi$ is the Euler class of
$TM$.
Hence, by \eqref{eq:cohomology}, there exists a unique class
$m(\phi)\in H^1(M;\R)/\im(\id-\phi^*)$ satisfying
$
d\,m(\phi) = [\ophi]-(\area_\omega(M)/\chi(M))\cdot c_\phi,
$
where $\chi(M)$ denotes the Euler characteristic of $M$.
Therefore, $\phi$ is monotone if and only if $m(\phi)=0$.
 
Because $c_\phi$ controls the index, or expected dimension, of moduli
spaces of holomorphic curves under change of homology class and $\omega_\phi$
controls their energy under change of homology class, the monotonicity condition
ensures that the energy is constant on the index one components of the moduli space,
which implies compactness and, as a  corollary, finite count  in a  differential 
of the Floer complex.

We recall the fundamental properties of $\symp^m(M,\omega)$ from \cite{S,G}.
Let  $\diff^+(M)$ denotes  the group of orientation preserving
diffeomorphisms of $M$.
\smallskip\\
(Identity) $\id_M \in \symp^m(M,\omega)$.
\smallskip\\
(Naturality)\label{page:natur}
If $\phi\in\symp^m(M,\omega),\psi\in\diff^+(M)$, then
$\psi^{-1}\phi\psi\in\symp^m(M,\psi^*\omega)$.
\smallskip\\
(Isotopy)
Let $(\psi_t)_{t\in I}$ be an isotopy in $\symp(M,\omega)$, i.e. a smooth
path with $\psi_0=\id$.
Then
$
m(\phi\circ\psi_1)=m(\phi)+[\flux(\psi_t)_{t\in I}]
$
in $H^1(M;\R)/\im(\id-\phi^*)$; see \cite[Lemma 6]{S}.
For the definition of the flux homomorphism see \cite{MS}.
\smallskip\\
(Inclusion)
The inclusion $\symp^m(M,\omega)\inclusion\diff^+(M)$ is a homotopy
equivalence.
\smallskip\\
(Floer homology)
To every $\phi\in\symp^m(M,\omega)$ symplectic Floer homology
theory assigns a $\Z_2$-graded vector space $\HF(\phi)$ over $\Z_2$, with an
additional multiplicative structure, called the quantum cap product,
$
H^*(M;\Z_2)\otimes\HF(\phi)\To\HF(\phi).
$
For $\phi=\id_M$ the symplectic Floer homology $\HF(\id_M)$ are  canonically isomorphic to ordinary homology  $H_*(M;\Z_2)$ and quantum cap product agrees with the ordinary cap product.
Each $\psi\in\diff^+(M)$ induces an isomorphism
$\HF(\phi)\cong\HF(\psi^{-1}\phi\psi)$ of $H^*(M;\Z_2)$-modules.
\smallskip\\
(Invariance)
If $\phi,\phi'\in\symp^m(M,\omega)$ are isotopic, then
$\HF(\phi)$ and $\HF(\phi')$ are naturally isomorphic as
$H^*(M;\Z_2)$-modules.
This is proven in \cite[Page 7]{S}. Note that every Hamiltonian perturbation
of $\phi$ (see \cite{ds}) is also in $\symp^m(M,\omega)$.
\smallskip\\
Now let $g$ be a mapping class of $M$, i.e. an isotopy class of $\diff^+(M)$.
Pick an area form $\omega$ and a
representative $\phi\in\symp^m(M,\omega)$ of $g$.
Then $\HF(\phi)$ is an invariant of $g$, which is
denoted by $\HF(g)$. Note that $\HF(g)$ is independent of the choice of an area
form $\omega$ by Moser's isotopy theorem \cite{Mo} and naturality of Floer homology.

\subsubsection{Floer homology}

Let $\phi\in\symp(M,\omega)$.There are two ways of constructing Floer
homology detecting its fixed points, $\Fix(\phi)$. Firstly, the graph of $\phi$
is a Lagrangian submanifold of $M\times M,(-\omega)\times\omega)$ and its fixed points correspond to the intersection points of graph($\phi$) with the
diagonal $\Delta=\{(x,x)\in M\times M\}$. Thus we have the Floer homology of the Lagrangian intersection $\HF(M\times M,\Delta, graph (\phi))$.
This intersection is transversal if the fixed points of $\phi$ are nondegenerate, i.e. if 1 is not an eigenvalue of $d\phi(x)$, for $x\in\Fix(\phi)$.
The second approach was mentioned by Floer in \cite{Floer1} and presented with
details by Dostoglou and Salamon  in \cite{ds}.We follow here  Seidel's approach \cite{S} which, comparable with \cite {ds}, uses a larger
class of perturbations, but such that the perturbed action form is still
cohomologous to the unperturbed. As a consequence, the usual invariance of
Floer homology under Hamiltonian isotopies is extended to the stronger
property stated above.
Let now  $\phi\in\symp^m(M,\omega)$, i.e  $\phi$ is  monotone.
Firstly, we give  the definition of $\HF(\phi)$  in the special case
where all the fixed points of $\phi$ are non-degenerate, i.e. for all $y\in\Fix(\phi)$, $\det(\id-\de\phi_y)\ne0$, and then
following Seidel´s approach  \cite{S} we consider general case when $\phi$ has degenerate fixed points.
 Let $\Of = \{ y \in C^{\infty}(\R,M)\,|\, y(t) = \phi(y(t+1)) \}$ be the twisted free loop space, which is also
the space of sections of $T_\phi \rightarrow S^1$. The action form is the
closed one-form $\alpha_\phi$ on $\Of$ defined by
$$
\alpha_\phi(y) Y = \int_0^1 \omega(dy/dt,Y(t))\,dt.
$$
where $y\in\Of$ and $Y\in T_y\Of$, i.e. $Y(t)\in T_{y(t)}M$ and
$Y(t)=\de\phi_{y(t+1)}Y(t+1)$ for all $t\in\R$.

The tangent bundle of any symplectic manifold admits an almost complex structure $ J:TM\To TM$ which is compatible with $\omega$ in sense that $(v,w)=\omega(v,Jw)$ defines a Riemannian metric.
 Let $J=(J_t)_{t \in \R}$
be a smooth path of $\omega$-compatible almost  complex structures on
$M$ such that $J_{t+1}=\phi^*J_t$.
If $Y,Y'\in T_y\Of$, then
$\int_0^1\omega(Y'(t),J_t Y(t))\de t$ defines a metric
on the loop space $\Of$. So  the critical points of $\alpha_\omega$ are the constant paths in  $\Of$ and  hence the fixed points of $\phi$. The negative gradient lines of $\alpha_\omega$ with respect to the
metric above  are solutions of the partial differential equations
with boundary conditions
\begin{equation}\label{eq:corbit}
\left\{\begin{array}{l}
u(s,t) = \phi(u(s,t+1)), \\
\p_s u + J_t(u)\p_t u = 0, \\
\lim_{s\to\pm\infty}u(s,t) \in \Fix(\phi)
\end{array}\right.
\end{equation}
These are exactly  Gromov's pseudoholomorphic  curves \cite{Gromov}.

For $y^\pm\in\Fix(\phi)$, let $\M(y^-,y^+;J,\phi)$ denote the space of smooth maps $u:\R^2\to M$ which satisfy the  equations \eqref{eq:corbit}.
Now to every $u\in\M(y^-,y^+;J,\phi)$ we  associate a Fredholm operator
$\De_u$ which linearizes (\ref{eq:corbit}) in suitable Sobolev spaces. The
index of this operator is given by the so called Maslov index $\mu(u)$,
which satisfies $\mu(u)=\deg(y^+)-\deg(y^-)\text{ mod }2$, where  $(-1)^{\deg y}=\sign(\det(\id-\de\phi_y))$. We have no bubbling, since for surface
$\pi_2(M)=0$. For a generic
$J$, every $u\in\M(y^-,y^+;J,\phi)$ is regular, meaning that $\De_u$ is onto.
Hence, by the implicit function theorem, $\M_k(y^-,y^+;J,\phi)$ is
a smooth $k$-dimensional manifold, where $\M_k(y^-,y^+;J,\phi)$ denotes the
subset of those $u\in\M(y^-,y^+;J,\phi)$ with $\mu(u)=k\in\Z$.
Translation of the $s$-variable defines a free $\R$-action on 1-dimensional
manifold $\M_1(y^-,y^+;J,\phi)$ and hence the quotient is a discrete set of points. The energy of a map $u:\R^2\to M$ is given by $
E(u) = \int_{\R}\int_0^1 \omega\big(\p_tu(s,t),J_t\p_tu(s,t)\big)\,\de t\de s$  for all $y\in\Fix(\phi)$.  P.Seidel has proved  in \cite{S} that if $\phi$ is monotone, then the energy is constant on each $\M_k(y^-,y^+;J,\phi)$.
Since all fixed points of $\phi$ are nondegenerate the set  $\Fix(\phi)$ is a finite set and the
$\Z_2$-vector space  $
\CF(\phi) := \Z_2^{\#\Fix(\phi)}$
admits a $\Z_2$-grading with $(-1)^{\deg y}=\sign(\det(\id-\de\phi_y))$,
for all $y\in \Fix(\phi)$.
The boundedness of the energy  $E(u)$   for monotone  $\phi$  implies that   the  0-dimensional  quotients   $\M_1(y_-,y_+,J,\phi)/\R$   are actually finite sets. Denoting by  $n(y_-,y_+)$  the number of
points mod 2 in each of them, one defines a differential  $\partial_{J}:
CF_*(\phi) \rightarrow CF_{* + 1}(\phi)$  by $\partial_{J}y_- =
\sum_{y_+} n(y_-,y_+) {y_+}$.  Due to gluing theorem  this Floer boundary operator satisfies  $\partial_{J} \circ
\partial_{J} = 0$.  For gluing  theorem to hold one needs again the  boundedness of the energy $E(u)$ .  It follows that  $ (\CF(\phi),\partial_{J})$  is a chain complex  and its homology is by definition the Floer homology of  $\phi$   denoted $HF_*(\phi)$. It  is independent of $J$ and is an invariant of $\phi$.

If $\phi$ has degenerate fixed points one needs to perturb equations
\eqref{eq:corbit} in order to define the Floer homology. Equivalently, one
could say that the action form needs to be perturbed.
 The necessary analysis is  given in \cite{S}, it  is essentially  the same as in the slightly different
situations considered in \cite{ds}. But  Seidel's approach also differs from the usual one in \cite{ds}. He uses a larger
class of perturbations, but such that the perturbed action form is still
cohomologous to the unperturbed.

\subsection{Nielsen numbers and Floer homology\label{sec:numbers}}

\subsubsection{Periodic diffeomorphisms}

\begin{lemma}\label{lemma:jiang}\cite{J}
Let $\phi$  a non-trivial orientation
preserving periodic diffeomorphism of a compact connected surface $M$
of Euler characteristic $\chi(M) \leq0$. Then each fixed point class of $\phi$
consists of a single point which has index 1.
\end{lemma}
There are  two criteria for monotonicity which we use later on.
 Let $\omega$ be an area form on $M$ and
$\phi\in\symp(M,\omega)$. 
\begin{lemma}\label{lemma:monotone1}\cite{G}
Assume that every class $\alpha\in\ker(\id-\phi_*)\subset H_{1}(M;\Z)$ is
represented by a map $\gamma:S\to\Fix(\phi)$, where $S$ is a compact oriented
1-manifold. Then $\phi$ is monotone.
\end{lemma}
\begin{lemma}\label{lemma:monotone2}\cite{G}
If $\phi^k$ is monotone for some $k>0$, then $\phi$ is monotone.
If $\phi$ is monotone, then $\phi^k$ is monotone for all $k>0$. 
\end{lemma}

We shall say  that $\phi:M\rightarrow M$ is a  periodic map of  period $m$,
if $\phi^m$ is  the identity map $\id_M:M\rightarrow M$.

\begin{theorem}\label{thm:main}\cite{ff}
If $\phi$ is  a non-trivial orientation
preserving periodic diffeomorphism of a compact connected surface $M$
of Euler characteristic $\chi(M) \leq  0$, then $\phi$ is monotone symplectomorphism  with
respect to some $\phi$-invariant area form and
$$
\HF(\phi) \cong
\Z_2^{N (\phi)}
$$
where  $ N (\phi)$ denotes the Nielsen  number of  $\phi$.
\end{theorem}

\begin{proof}
Let  $\phi$ be  a periodic diffeomorphism  of least period $l$.
 First note that
if $\tilde{\omega}$ is an area form on $M$, then area form
$\omega:=\sum_{i=1}^\ell(\phi^i)^*\tilde{\omega}$ is 
 $\phi$-invariant, i.e. $\phi\in\symp(M,\omega)$. By periodicity of  $\phi$,
 $\phi^l$ is  the identity map $\id_M:M\rightarrow M$.
Then from Lemmas \ref{lemma:monotone1}  and \ref{lemma:monotone2} it follows that $\omega$ can be chosen such that $\phi\in\symp^m(M,\omega)$.

 Lemma \ref{lemma:jiang}  implies that every
$y\in\Fix(\phi)$ forms a different fixed point class of
$\phi$, so $ \#\Fix(\phi)= N (\phi)$.
This has an immediate consequence for the Floer complex $(\CF(\phi),\p_{J})$
with respect to a generic $J=(J_t)_{t\in\R}$.
If $y^\pm\in\Fix(\phi)$ are in different fixed point classes, then
$\M(y^-,y^+;J,\phi)=\emptyset$. This follows from the first equation in
\eqref{eq:corbit}.
Then the boundary map in  the Floer complex is zero $\partial_{J}=0$ and
$\Z_2$-vector space  $
\CF(\phi) := \Z_2^{\#\Fix(\phi)}=\Z_2^{N(\phi)}$. This immediately  implies
$
\HF(\phi) \cong
\Z_2^{N (\phi)}
$ and $ \dim\HF(\phi)= N (\phi)$.

\end{proof}

\subsubsection{Algebraically finite mapping classes}

 A mapping class of $M$ is called  algebraically finite if it
does not have any pseudo-Anosov components in the sense of Thurston's
theory of surface diffeomorphism.The term algebraically finite goes back to J. Nielsen \cite{N}.\\
 In  \cite{G}  the diffeomorphisms  of finite type  were defined . These are
special representatives of algebraically finite mapping classes
adopted to the symplectic geometry. 
\begin{dfn}\label{def:ftype}\cite{G}
We call $\phi\in\diff_+(M)$ of {\bf finite type} if the following holds.
There is a $\phi$-invariant finite union $N\subset M$ of disjoint
non-contractible annuli such that:
\smallskip\\
(1) $\phi|M \setminus N$ is periodic, i.e. there exists
$\ell>0$ such that $\phi^\ell|M \setminus N=\id$.
\smallskip\\
(2) Let $N'$ be a connected component of $N$ and $\ell'>0$ be the
smallest integer such that $\phi^{\ell'}$ maps $N'$ to itself. Then
$\phi^{\ell'}|N'$ is given by one of the following two models with respect to
some coordinates $(q,p)\in I\times S^1$:
\medskip\\
\begin{minipage}{4cm}
(twist map)
\end{minipage}
\begin{minipage}{6cm}
$(q,p)\Mapsto(q,p-f(q))$
\end{minipage}
\medskip\\
\begin{minipage}{4cm}
(flip-twist map)
\end{minipage}
\begin{minipage}{6cm}
$(q,p)\Mapsto(1-q,-p-f(q))$,
\end{minipage}
\medskip\\
where $f:I\to \R$ is smooth and strictly monotone.
A twist map is called  positive or negative,
if $f$ is increasing or decreasing.
\smallskip\\
(3) Let $N'$ and $\ell'$ be as in (2).
If $\ell'=1$ and $\phi|N'$ is a twist map, then $\im(f)\subset[0,1]$,
i.e. $\phi|\text{int}(N')$ has no fixed points.
\smallskip\\
(4) If two connected components of $N$ are homotopic, then the corresponding local
models of $\phi$ are either both positive or both negative twists.

\end{dfn}
The term flip-twist map is taken from \cite{JG}.

By  $M_{\id}$ we denote the union of the components
of $M\setminus\text{int}(N)$, where $\phi$ restricts to the identity.

The next  lemma describes the set of fixed point classes of $\phi$. It
is a special case of a theorem by B.~Jiang and J.~Guo~\cite{JG}, which gives
for any mapping class a representative that realizes its
Nielsen number.
\begin{lemma}[Fixed point classes]\label{lemma:fclass}\cite{JG}
Each fixed point class of $\phi$ is either a connected component of $M_{\id}$ or
consists of a single fixed point. A fixed point $x$ of the second type satisfies
$\det(\id-\de\phi_x)>0$.
\end{lemma}

The monotonicity of diffeomorphisms of finite type
was  investigated in details in \cite{G}.
Let $\phi$ be a diffeomorphism of finite type and
$\ell$ be as in (1).
Then $\phi^\ell$ is the product of (multiple)  Dehn twists along $N$.
Moreover, two parallel Dehn twists have the same sign, by (4). We say that
$\phi$ has  uniform twists, if $\phi^\ell$ is the product of only
positive, or only negative Dehn twists.
\smallskip\\
Furthermore, we denote by $\ell$ the smallest positive integer such that
$\phi^\ell$ restricts to the identity on $M\setminus N$.

If $\omega'$ is an area form on $M$ which is the standard form
$\de q\wedge\de p$ with respect to the $(q,p)$-coordinates on $N$, then
$\omega:=\sum_{i=1}^\ell(\phi^i)^*\omega'$ is
standard on $N$ and $\phi$-invariant, i.e. $\phi\in\symp(M,\omega)$.
To prove that $\omega$ can be chosen such that $\phi\in\symp^m(M,\omega)$,
Gautschi distinguishes two cases: uniform and non-uniform twists. In the first case he proves the following stronger statement.
\begin{lemma}\label{lemma:monotone3}\cite{G}
If $\phi$ has uniform twists and $\omega$ is a $\phi$-invariant area
form, then $\phi\in\symp^m(M,\omega)$.
\end{lemma}

In the non-uniform case, monotonicity  does not
hold for arbitrary $\phi$-invariant area forms.
\begin{lemma}\label{lemma:monotone4}\cite{G}
If $\phi$ does not have uniform twists, there exists a $\phi$-invariant
area form $\omega$ such that $\phi\in\symp^m(M,\omega)$. Moreover,
$\omega$ can be chosen such that it is the standard form $\de q\wedge\de p$ on
$N$.
\end{lemma}

\begin{theorem}\label{thm:main0}\cite{ff}
If $\phi$ is  a  diffeomorphism of finite type  of a compact connected surface $M$
of Euler characteristic $\chi(M) <  0$ and if  $\phi$ has  only isolated fixed points , then $\phi$ is monotone with
respect to some $\phi$-invariant area form and
$$
\HF(\phi) \cong
\Z_2^{N (\phi)}, 
$$
where  $ N (\phi)$ denotes the Nielsen  number of  $\phi$.
\end{theorem}
\begin{proof}
From Lemmas \ref{lemma:monotone3}  and \ref{lemma:monotone4} it follows that $\omega$ can be chosen such that $\phi\in\symp^m(M,\omega)$.
 Lemma \ref{lemma:fclass}  implies that every
$y\in\Fix(\phi)$ forms a different fixed point class of
$\phi$, so $ \#\Fix(\phi)= N (\phi)$.
This has an immediate consequence for the Floer complex $(\CF(\phi),\p_{J})$
with respect to a generic $J=(J_t)_{t\in\R}$.
If $y^\pm\in\Fix(\phi)$ are in different fixed point classes, then
$\M(y^-,y^+;J,\phi)=\emptyset$. This follows from the first equation in
\eqref{eq:corbit}.
Then the boundary map in  the Floer complex is zero $\partial_{J}=0$ and
$\Z_2$-vector space  $
\CF(\phi) := \Z_2^{\#\Fix(\phi)}=\Z_2^{N(\phi)}$. This immediately  implies
$
\HF(\phi) \cong
\Z_2^{N (\phi)}
$.

\end{proof}

\begin{theorem}\label{th:ga}\cite{G}
Let  $\phi$ be  a diffeomorphism of finite type, then $\phi$ is monotone with
respect to some $\phi$-invariant area form and
$$
\HF(\phi) =
H_*(M_{\id},\p_{M_{\id}};\Z_2)\oplus
\Z_2^{L(\phi|M\setminus M_{\id})}.
$$
Here, $L$ denotes the Lefschetz number.
\end{theorem}

\begin{proof}

The main idea of the proof  is a separation mechanism for Floer connecting
orbits.
Together with the topological separation of fixed points discussed in
theorem \ref{thm:main0} , it allows us to compute the Floer homology of
diffeomorphisms of finite type.
There exists a function $H: M\to\R$ 
such that $H|\text{int}(M_{id})$ is a Morse function,
meaning that all the critical points are non-degenerate and $H|(M\setminus M_{id})=0$. 
Let $(\psi_t)_{t\in\R}$ denote the Hamiltonian flow generated by $H$ with
respect to the fixed area form $\omega$ and set
$
\Phi:=\phi\circ\psi_1.$ Then $
\Fix\Phi = \big(\crit(H)\cap M_{id}\big)\cup\big(\Fix\phi\setminus M_{id}\big).
$
In particular, $\Phi$ only has non-degenerate fixed points.
Let $N_0\subset M_{id}$ be a collar neighborhood of
$\p M_{id}$.
Let $x^-,x^+\in\Fix\Phi\cap M_{id}$ be in the same connected component of
$M_{id}$.
If $u\in\M(x^-,x^+;J,\Phi)$, then $\im u \subset M_\delta$,
where $M_\delta$ denotes the $\delta$-neighborhood of
$M_{id}\setminus N_0$  with respect to any of the metrics $\omega(.,J_t.)$ 
\cite{S,G}
 Moreover, lemma \ref{lemma:fclass}  implies that every
$y\in\fix\phi\setminus M_{id}$ forms a different fixed point class of
$\Phi$.
This has an immediate consequence for the Floer complex $(\CF(\Phi),\p_{J})$
with respect to a generic $J=(J_t)_{t\in\R}$. Namely,
$(\CF(\Phi),\p_{J})$ splits into the subcomplexes $(\Cc_1,\p_1)$ and
$(\Cc_2,\p_2)$, where $\Cc_1$ is generated by $\crit(H)\cap M_{id}$
and $\Cc_2$ by $\fix\phi\setminus M_{id}$. Moreover, $\Cc_2$ is graded by~$0$
and $\p_2=0$ \cite{G}.The homology of $(\Cc_1,\p_1)$ is isomorphic to
$H_*(M_{id},\p_+M_{id};\Z_2)$ \cite{S,G}. So
$
\HF(\phi) \cong
H_*(M_{id},\p_+M_{id};\Z_2)\oplus
\Z_2^{\#\Fix\phi|M\setminus M_{id}}.
$
Since every fixed point of $\phi|M\setminus M_{id}$ has fixed point
index~1, the Lefschetz fixed point formula implies that
$
\#(\Fix\phi\setminus M_{id}) = L(\phi|(M\setminus M_{id})).
$

\end{proof}

\begin{rk}
 In  the  theorem \ref{thm:main0} the  set $ M_{\id}$ is empty and  every fixed point of $\phi$ has fixed point index 1 \cite{JG}.
The Lefschetz fixed point formula implies that
$\#\Fix \phi=N(\phi)=L(\phi)$ .
So, theorem \ref{thm:main0}  follows also  from theorem  \ref{th:ga}.

\end{rk}

\subsubsection{ Anosov  diffeomorphisms of 2-dimensional  torus}

The connection between Nielsen number and the dimension of symplectic Floer homology remains true for genus 1 surface  and Anosov diffeomorphism .
\begin{theorem}\cite{ff}
If $\phi$ is a  Anosov  diffeomorphism
 of a 2-dimensional torus $T^2$,
 then $\phi$ is symplectic  and 
$$
\HF(\phi) \cong
\Z_2^{N (\phi)}, 
$$
where  $ N (\phi)=|\det (E-\phi_*)|$ denotes the Nielsen  number of  $\phi$
and $\phi_*$ is an  induced homomorphism on the fundamental group of
    $T^2$.
\end{theorem}
\begin{proof}
Hyperbolicity of $\phi$  means that the covering linear map
 $\tilde \phi  : R^2 \rightarrow R^2 $, $\det\tilde \phi=1$  has no eigenvalue of modulus one. The Anosov  diffeomorphism of a 2-dimensional torus $T^2$ is area preserving so symplectic.
In fact, the covering map $\tilde{\phi} $ has a unique fixed point, which is the origin;
hence, by the covering homotopy theorem , the fixed points of $\phi$ are
pairwise Nielsen  nonequivalent.The index of each Nielsen  fixed point  class, consisting of one fixed
point, coincides with its Lefschetz index, and by the hyperbolicity of $\phi$, the later
is not equal to zero.Thus the Nielsen number $ N(\phi)= \#\Fix(\phi)$.
If $y^\pm\in\Fix(\phi)$ are in different Nielsen  fixed point classes, then 
$\M(y^-,y^+;J,\phi)=\emptyset$. This follows from the first equation in
\eqref{eq:corbit}.
Monotonicity condition trivially  follows from  the isomorphism
$ H_2(M,\R)\cong H_2(T_\phi,\R)$.
Then the boundary map in  the Floer complex is zero $\partial_{J}=0$ and 
$\Z_2$-vector space  $
\CF(\phi) := \Z_2^{\#\Fix(\phi)}=\Z_2^{N(\phi)}$.

 This immediately  implies
$
\HF(\phi) \cong
\Z_2^{N (\phi)}
$ .

\end{proof}

\section{Symplectic zeta functions and asymptotic invariant} \label{sec:zeta}

\subsection{Symplectic zeta functions}
Let $ Mod_M= \pi_0(Diff^+(M))$ be the mapping class group of a closed connected oriented surface $M$ of genus $\geq 2$. Pick an everywhere positive two-form $\omega$ on $M$. A isotopy  theorem of Moser \cite{Mo} says that each  mapping class of  $g \in \Gamma$,  i.e. an isotopy class of $Diff^+(M)$,  admits representatives which preserve $\omega$. Due  to Seidel\cite{S} we can
pick  a monotone  representative $\phi\in\symp^m(M,\omega)$ of $g$.
Then $\HF(\phi)$ is an invariant of $g$, which is
denoted by $\HF(g)$. Note that $\HF(g)$ is independent of the choice of an area
form $\omega$ by Moser's  theorem  and naturality of Floer homology.
By  lemma  \ref{lemma:monotone2}  symplectomorphisms  $\phi^n$
are  also monotone for all $n>0$.
Taking a dynamical point of view,
 we consider the iterates of monotone symplectomorphism  $\phi$
and  define the first symplectic  zeta function of $\phi$ \cite{ff}
 as the following power series:
$$
 \chi_\phi(z) =
 \exp\left(\sum_{n=1}^\infty \frac{\chi(\HF(\phi^n))}{n} z^n \right),
$$
where $\chi(\HF(\phi^n))$ is the Euler characteristic of Floer homology complex
of $\phi^n$.
Then $ \chi_\phi(z)$ is an invariant of $g$, which we
denote by $ \chi_g(z)$.
Let us consider  the  Lefschetz  zeta function
 $$
  L_\phi(z) := \exp\left(\sum_{n=1}^\infty \frac{L(\phi^n)}{n} z^n \right),
$$
   where
 $
   L(\phi^n) := \sum_{k=0}^2 (-1)^k \tr\Big[\phi_{*k}^n:H_k(M;\Q)\to H_k(M;\Q)\Big]
 $
 is the Lefschetz number of $\phi^n$. 
\begin{theorem}\label{thm:lef}\cite{ff}
Symplectic zeta function $ \chi_\phi(z)$ is a rational function of $z$ and
$$
 \chi_\phi(z)= L_\phi(z)=\prod_{k=0}^2
          \det\big(I-\phi_{*k}.z\big)^{(-1)^{k+1}}.
$$
\end{theorem}

\begin{proof} 
If for every $n$ all the fixed points of $\phi^n$ are non-degenerate, i.e. for all $x\in\Fix(\phi^n)$, $\det(\id-\de\phi^n(x))\ne0$, then we have( see section
\ref{sec:floer}): 
$\chi(\HF(\phi^n))=\sum_{x=\phi^n(x)} \sign(\det(\id-\de\phi^n(x)))=L(\phi^n).
$
If we have degenerate fixed points one needs   to perturb equations
\eqref{eq:corbit} in order to define the Floer homology.
 The necessary analysis is given in \cite{S}  is essentially  the same as in the slightly different situations considered in \cite{ds}, where the  above connection between the Euler characteristic and the Lefschetz number
 was firstly  established.
\end{proof}

In \cite{ff} we have  defined  the second symplectic zeta function for 
 monotone symplectomorphism  $\phi$  as the following power series:
$$
 F_\phi(z) = 
 \exp\left(\sum_{n=1}^\infty \frac{\dim\HF(\phi^n)}{n} z^n \right).
$$
Then $ F_\phi(z)$ is an invariant of mapping class  $g$, which we
denote by $ F_g(z)$.

Motivation for this definition was  the  theorem \ref{thm:main} and nice analytical
properties of the Nielsen zeta function $$N_\phi(z) =
 \exp\left(\sum_{n=1}^\infty \frac{N(\phi^n)}{n} z^n \right)$$, see \cite{fv,fl,pf,FelshB}. 
We denote the  numbers  $\dim\HF(\phi^n) $ by $N_n$. Let $ \mu(d), d \in \N$,
be the M\"obius function.

 \begin{theorem}\cite{ff}
Let $\phi$ be  a non-trivial orientation
preserving periodic diffeomorphism  of least period $m$ of a compact connected surface $M$
of Euler characteristic $\chi(M) < 0$  . Then the
  zeta function $ F_\phi(z)$ is  a radical of a rational function and 
$$
 F_\phi(z) =\prod_{d\mid m}\sqrt[d]{(1-z^d)^{-P(d)}},
$$
 where the product is taken over all divisors $d$ of the period $m$, and $P(d)$ is the integer
$  P(d) = \sum_{d_1\mid d} \mu(d_1)N_{d/ d_1} .  $
\end{theorem}

\begin{rk}
Given a symplectomorphism $\phi$ of surface $M$, one can form
the symplectic mapping torus
$M^4_{\phi}=T^3_{\phi}\rtimes S^1$, where $T^3_{\phi}$ is  usual mapping torus
.
Ionel and Parker \cite{IP} have computed the degree zero Gromov invariants
\cite{MS1}(these are built from the invariants of Ruan and Tian)  of
$M^4_{\phi}$ and of fiber sums of the $M^4_{\phi}$ with other symplectic manifolds. This is done by expressing the Gromov invariants in terms of the
Lefschetz zeta function $ L_\phi(z)$ \cite{IP}. The result is a large set of interesting non-Kahler
symplectic manifolds with computational ways of distinguishing them. In dimension four this gives a symplectic construction of the exotic elliptic surfaces of Fintushel and Stern \cite{FS}. This construction arises from knots.
Associated to each fibered knot $K$ in $S^3$ is a Riemann surface $M$ and a
monodromy diffeomorphism $f_K$ of $M$. Taking $\phi=f_K$ gives symplectic
4-manifolds ${M^4}_\phi(K)$ with Gromov invariant
$Gr({M^4}_\phi(K))= A_K(t)/(1-t)^2=L_\phi(t)$, where $ A_K(t)$ is the Alexander polynomial of knot $K$. Next, let $E^4(n)$ be the simply-connected minimal elliptic surface with fiber $F$ and canonicla divisor $k=(n-2)F$. Forming the fiber sum $E^4(n,K)=E^4(n)\#_{(F=T^2)}{M^4}_\phi(K)$ we obtain a
symplectic manifold homeomorphic to $E^4(n)$.
Then for $n\geq 2$ the Gromov and Seiberg-Witten invariants of $E^4(K)$
are $ Gr(E^4(n,K))=SW(E^4(n,K))=A_K(t)(1-t)^{n-2}$ \cite{FS, IP}. 
Thus fibered knots with distinct Alexander polynomials give rise to symplectic manifolds $E^4(n,K)$ which are homeomorphic but not diffeomorphic. In particular, there are infinitely many distinct symplectic 4-manifolds homeomorphic
to $E^4(n)$ \cite{FS} .

In higher dimensions it gives many examples of manifolds which are diffeomorphic but not equivalent as symplectic manifolds.
  Theorem \ref{thm:lef} implies that the Gromov invariants of $M^4_{\phi}$ are related to symplectic Floer homology  of $\phi $ via zeta
function
 $\chi_\phi(z)= L_\phi(z)$. We hope that  the second symplectic zeta function $
 F_\phi(z)$ give rise to a new invariant of symplectic 4-manifolds \cite{FS}.

\end{rk}

\subsection{Topological entropy and the Nielsen numbers}

   A basic relation between Nielsen numbers and topological entropy $h(f)$ \cite{KH} was found by N. Ivanov
   \cite{i1}. We present here a very short proof of Jiang  of the Ivanov's  inequality.
\begin{lemma}\label{lemma:ent}\cite{i1}
$$
h(f) \geq \limsup_{n} \frac{1}{n}\cdot\log N(f^n)
 $$
\end{lemma}
\begin{proof}
 Let $\delta$ be such that every loop in $X$ of diameter $ < 2\delta $ is contractible.
 Let $\epsilon >0$ be a smaller number such that $d(f(x),f(y)) < \delta $ whenever $ d(x,y)<2\epsilon $. Let $E_n \subset X $ be a set consisting of one point from each essential fixed point class of $f^n$. Thus $ \mid E_n \mid =N(f^n) $. By the definition of $h(f)$, it suffices
 to show that $E_n$ is $(n,\epsilon)$-separated.
 Suppose it is not so. Then there would be two points $x\not=y \in E_n$ such that $ d(f^i(x), f^i(y)) \leq \epsilon$ for $o\leq i< n$ hence for all $i\geq 0$. Pick a path $c_i$ from $f^i(x)$ to
 $f^i(y)$ of diameter $< 2\epsilon$ for $ o\leq i< n$ and let $c_n=c_0$. By the choice of $\delta$
 and $\epsilon$ ,  $f\circ c_i \simeq c_{i+1} $ for all $i$, so $f^n\circ c_0\simeq c_n=c_0$. such that
 This means $x,y$ in the same fixed point class of $f^n$, contradicting the construction of $E_n$.

\end{proof}

 This inequality is remarkable in that it does not require smoothness of the map and provides a common lower bound for the topological entropy of all maps in a homotopy class.

We recall Thurston classification theorem for homeomorphisms of surfase $M$
of genus $\geq 2$.

\begin{theorem}\label{thm:thur}\cite{Th}
Every homeomorphism $\phi: M\rightarrow M $ is isotopic to a homeomorphism $f$
such that either\\
(1) $f$ is a periodic map; or\\
(2) $f$ is a pseudo-Anosov map, i.e. there is a number $\lambda >1$(stretching factor) and a pair of transverse measured foliations $(F^s,\mu^s)$ and $(F^u,\mu^u)$ such that $f(F^s,\mu^s)=(F^s,\frac{1}{\lambda}\mu^s)$ and $f(F^u,\mu^u)=(F^u,\lambda\mu^u)$; or\\
(3)$f$ is reducible map, i.e. there is a system of disjoint simple closed curves $\gamma=\{\gamma_1,......,\gamma_k\}$ in $int M$ such that $\gamma$ is invariant by $f$(but $\gamma_i$ may be permuted) and $ \gamma$ has a $f$-invariant
tubular neighborhood $U$  such that each component of $M\setminus U$  has negative
Euler characteristic and on each(not necessarily connected) $f$-component of
$M\setminus U$, $f$ satisfies (1) or (2).

\end{theorem}
The map $f$ above is called a  Thurston canonical form of $\phi$. In (3) it
can be chosen so that some iterate $f^m$ is a generalised Dehn twist on $U$.
Such a $f$ , as well as the $f$ in (1) or (2), will be called standard.
A key observation is that if $f$ is standard, so are all iterates of $f$.

\begin{lemma}\label{lem:flp}\cite{flp} Let $f$ be a pseudo-Anosov  homeomorphism with stretching factor $\lambda >1$ of surfase $M$
of genus $\geq 2$. Then
$$h(f)=log(\lambda)= \limsup_{n} \frac{1}{n}\cdot\log N(f^n)$$
\end{lemma}

\begin{lemma}\label{lem:j1}\cite{j1}
Suppose $f$ is a standard homeomorphism  of surfase $M$
of genus $\geq 2$ and  $\lambda$ is the largest stretching factor  of the  pseudo-Anosov pieces( $\lambda=1$ if there is no pseudo-Anosov piece).
Then
$$h(f)=log(\lambda)= \limsup_{n} \frac{1}{n}\cdot\log N(f^n)$$
\end{lemma}

\subsection{Asymptotic invariant}

The growth rate of a sequence $a_n$ of complex numbers is defined
by
 $$
 \grow ( a_n):= max \{1,  \limsup_{n \rightarrow \infty} |a_n|^{1/n}\}
$$

which could be infinity. Note that $\grow(a_n) \geq 1$ even if all $a_n =0$.
When  $\grow(a_n) > 1$, we say that the sequence $a_n$ grows exponentially.
\begin{dfn}
 We define the asymptotic invariant $ F^{\infty}(g)$ of  mapping class   $g \in Mod_M = \pi_0(Diff^+(M))$ to be the growth rate of the sequence
$\{a_n=\dim\HF(\phi^n)\}$ for  a monotone  representative $\phi\in\symp^m(M,\omega)$ of $g$:
$$ F^{\infty}(g):=\grow(\dim\HF(\phi^n)) $$
\end{dfn}
\begin{ex}
If $\phi$ is  a non-trivial orientation
preserving periodic diffeomorphism of a compact connected surface $M$
of Euler characteristic $\chi(M) <  0$ , then the  periodicity of the
sequence $\dim\HF(\phi^n)$ implies that  for the  corresponding  mapping class $g$ the  asymptotic invariant
$$ F^{\infty}(g):=\grow(\dim\HF(\phi^n))=1 $$

\end{ex}

\begin{ex} Let $\phi$ be  a monotone  diffeomorphism of finite type of a compact connected surface $M$
of Euler characteristic $\chi(M) <  0$ and $g$ a corresponding algebraically finite mapping class.
Let $U$ be the open regular neighborhood of the $k$ reducing
curves $\gamma_1,......,\gamma_k$ in the Thurston theorem, and $M_j$
be the component of $M\setminus U$.Let $F$ be a fixed point class of $\phi$.
Observe from \cite{JG} that if $F\subset M_j$, then $\ind(F, \phi)=\ind(F,\phi_j)$.So if $F$ is counted in $N(\phi)$ but not counted in $\sum_{j} N(\phi_j)$
, it must intersect $U$. But we see from \cite{JG} that a component
of $U$ can intersect at most 2 essential fixed point classes of $\phi$.
Hence we have $N(\phi)\leq \sum_{j} N(\phi_j)$.For the  monotone  diffeomorphism of finite type $\phi$ maps $\phi_j$ are  periodic.
Applying last  inequality to $\phi^n$ and using  remark   \ref{th:ga} we have

$$
0\leq\dim\HF(\phi^n)= \dim H_*(M^{(n)}_{\id},\p{M^{(n)}_{\id}};\Z_2)+N(\phi^n|(M\setminus M^{(n)}_{\id}))\leq
$$
$$
\leq \dim H_*(M^{(n)}_{\id},\p{M^{(n)}_{\id}};\Z_2)+N(\phi^n)\leq
$$
$$
 \dim H_*(M^{(n)}_{\id},\p{M^{(n)}_{\id}};\Z_2)+ \sum_{j} N((\phi_j)^n) +2k\leq Const
$$
by periodicity of $\phi_j$.
Taking the growth rate in $n$, we get
that asymptotic invariant  $ F^{\infty}(g)=1$.
\end{ex}

\begin{rk}\label{conj:1}
For pseudo-Anosov  mapping class   $g \in Mod_M = \pi_0(Diff^+(M))$ we have
$$
\dim HF_*(g) >  N(g), \,\,   F^{\infty}(g) > \limsup_{n \rightarrow \infty} |N(g^n)|^{1/n}=h(\psi)=\lambda > 1,
$$
where  $ N (g)$ denotes the Nielsen  number of  $g$ and $\psi$ is a standard(Thurston canonical form) representative of $g$.

\end{rk}

\begin{rk}
For any  mapping class   $g \in Mod_M = \pi_0(Diff^+(M))$
we have 
$$
\dim HF_*(g) \geq   N(g),\,\,  F^{\infty}(g) \geq \lambda,
$$
 where  
  $\lambda $ is the largest stretching factor of
pseudo-Anosov pieses of  a standard( Thurston canonical form) representative 
of $g$ ( $\lambda:=1 $ if there is no pseudo-Anosov
piece).
\end{rk}

\section{ A generalisation of the  Poincar\'e - Birkhoff theorem and  Arnold conjecture.
 Concluding  remarks} \label{sec:problem}
 A natural generalization of the Poincar\'e- Birkhoff theorem concerns 
the estimation of the number  of fixed points of symplectomorphism $\phi \in \symp(M^{2n},\omega)$. Actually this question was already raised by Birkhoff in \cite{B}. He wrote with reference to the Poincar\'e - Birkhoff theorem:
``Up to present time no proper generalisation of this theorem to higher dimensions has been found, so that its applications remains limited to dynamical systems with two degrees of freedom \cite{B}, page 150.''

Based on the Poincar\'e - Birkhoff theorem Arnold formulated in the 1960s
his famous conjecture:
a Hamiltonian symplectomorphism( time-1 map of a time-dependent  Hamiltonian flow) should have at least as many fixed points as a function on the manifold have critical points.

Let $\phi: M^{2n}\to M^{2n}$ be a Hamiltonian symplectomorphism of a compact symplectic
manifold $(M^{2n},\omega)$. In the case when all fixed points of $\phi$ are  nondegenerate  the Arnold conjecture asserts
that $$\# Fix(\phi) \geq \dim H_*(M,\Q)=\sum_{k=0}^{2n} b_k(M),$$
where $2n=\dim M,  b_k(M)=\dim H_k(M,\Q)$.

The Arnold conjecture in the nondegenerate case has now been proved in full generality.
It  was first proved by Eliashberg \cite{eli}
for Riemann surfaces. For tori of arbitrary dimension it  was proved in the celebrated paper by Conley and Zehnder\cite{conzeh}.
 The most important breakthrough was Floer's
proof of the Arnold conjecture in the nondegenerate case  for monotone symplectic manifolds \cite{Floer}.
His proof was based on Floer homology.
His method has been pushed through by Fukaya-Ono\cite{fukono}, Liu-Tian\cite{liutian} and Hofer-Salamon\cite{hofsal} to establish the nondegenerate case of the Arnold conjecture for all symplectic manifolds.

Some progress  has been made with the original Arnold  conjecture by Rudyak \cite{R},
using a development of Lusternik - Schnirelman theory called category weight.
Using these ideas it was proved in \cite{RO} that the original conjecture 
holds in the case where degeneracies are allowed, provided that both classes
$ [\omega ] $ and $c_1$ vanish on $\pi_2(M^{2n})$.

The Hamiltonian symplectomorphism $\phi$ is isotopic to identity map $id_M$.
In this case all fixed points  $\phi$ are in the same Nielsen fixed point class. The  Nielsen number of $\phi$ is 0 or 1 depending on  Lefschetz number is 0 or not. So, the Nielsen number is very weak  invariant   to estimate the number of fixed points of $\phi$ for Hamiltonian symplectomorphism.
From another side, as we saw  in theorem \ref{thm:main},   for the nontrivial periodic
symplectomorphism $\phi$ of a  surface, the Nielsen number of $\phi$ gives an exact estimation from below for the number of nondegenerate fixed points of $\phi$. These  considerations lead us to the following question
\begin{que}
How to estimate the number of 
 fixed points  of    symplectomorphism (not necessary Hamiltonian)  $\phi \in \symp(M^{2n},\omega)$ which has only {\it nondegenerate} fixed points?
\end{que}

This question can be considered as a  generalisation of the  Poincar\'e - Birkhoff theorem and as a generalisation of the  Arnold conjecture in nondegenerate
case .

\subsection{Algebraically finite  mapping class}

 If $\psi$ is a diffeomorphism of finite type of surface $M$ then  $\psi \in \symp^m(M,\omega)$ for some $\psi$-invariant form  $\omega $.
Suppose now that symplectomorphism $\phi$ has only non-degenerate fixed points
and $\phi$ is Hamiltonian isotopic to $\psi$. Then $\phi \in \symp^m(M,\omega)$
and $\HF(\phi)$ is isomorphic to $\HF(\psi)$.
So,  from theorem \ref{th:ga}  it follows that
$$
\# Fix(\phi) \geq \dim\HF(\phi)=\dim\HF(\psi)=
$$
$$
=\dim H_*(M_{\psi=\id},\p{M_{\psi=\id}};\Z_2)+N(\psi|(M\setminus M_{\psi=\id}))=
$$
$$
=\sum_{k=0}^{2} b_k(M_{\psi=\id},\p{M_{\psi=\id}};\Z_2)+N(\psi|(M\setminus M_{\psi=\id}))
$$
This estimation can be considered as a generalisation of the Poincar\'e - Birkhoff theorem and as a generalisation of Arnold conjecture for symplectomorphism
with nondegenerate fixed points in algebraically finite mapping class,
becouse it implies  Arnold conjecture  for   $\psi=id$.
If  $\psi$ is nontrivial orientation preserving  periodic diffeomorphism
then by theorem \ref{thm:main} $\psi$ is a  monotone symplectomorphism and  by lemma \ref{lemma:jiang} it has only nondegenerate fixed points. The   theorem \ref{thm:main}  implies  an estimation 
$$
\# Fix(\phi) \geq \dim\HF(\phi)=\dim\HF(\psi)= N(\psi)
$$
This estimation can be considered as a  generalisation of  Poincar\'e - Birkhoff theorem  and Arnold conjecture for
a nontrivial  orientation-preserving  periodic mapping class.

\subsection{Pseudo-Anosov mapping class}

  For  pseudo-Anosov  ``diffeomorphism'' $\psi $ in given pseudo-Anosov mapping class $g$
 we also have, as in theorems \ref{thm:main},\ref{thm:main0}
and
 \ref{th:ga}, a topological separation of fixed points
 \cite{Th, JG, I}, i.e the Nielsen number of pseudo-Anosov ``diffeomorphism''
$\psi$   equals to the number of
fixed points of $\psi $  and there are  no connecting orbits between them.
But we have the following difficulties.
Firstly, a  pseudo-Anosov ``diffeomorphism'' is a smooth and  a symplectic
automorphism only on the complement of the singular set of a pair of transverse  measured foliations.
Secondly, in the case of a  pseudo-Anosov ``diffeomorphism''   we have  to deal with
fixed singular  points of index $1 -p$ where $p\geq 3$ is a number of prongs
of fixed singularity . Such  fixed points
are  degenerate from symplectic point of view  and therefore need a local perturbation. The following theorem can be considered as a  generalisation of  Poincar\'e - Birkhoff theorem   for a  symplectomorphism with nondegenerate fixed points in a given pseudo-Anosov mapping class.

\begin{theorem}\label{th:pa}(Cotton-Clay \cite{cot},  Fel'shtyn)
If $\phi$ is  a  symplectomorphism with nondegenerate fixed points in
given pseudo-Anosov mapping class $ \{ \phi \}=g=\{ \psi \} $,  then
 $$\HF(\phi)=\HF(g)  \cong \Z_2^{\sum_{x\in\Fix(\psi)}|\Ind(x)|}$$
and 
$$ \# Fix(\phi) \geq \dim\HF(\phi)=\dim\HF(g)=\sum_{x\in\Fix(\psi)}|\Ind(x)|,
$$
where $\psi$ is a canonical pseudo-Anosov representative of $g$. 
\end{theorem}

\begin{proof}
We describe main steps of the proof.
Firstly   we smooth   the singular  pseudo-Anosov
map  $\psi$ locally  near  the   singularities using Hamiltonian vector fields to get symplectomorphism $\hat\psi$ in  the given mapping class $g$
with $|\Ind(x)|$ nondegenerate fixed points in the small neighborhood of $x$.
The idea of such smoothing was suggested to the author by Kaoru Ono in 2005.
Full details of this  smoothing    can be  found  in \cite{cot}.
Then we show that all $|\Ind(x)|$ nondegenerate fixed points in the small neighborhood of $x$ are  in the same Nielsen fixed point class.
On the next step we prove that there are no pseudoholomorphic curves
of index 1 between these  nondegenerate fixed points. Then the boundary map in  the Floer complex of $\hat\psi$ is zero  and
$\Z_2$-vector space  $
\CF(\hat\psi) := \Z_2^{\#\Fix(\hat\psi)}=\Z_2^{\sum_{x\in\Fix(\psi)}|\Ind(x)|}$. This immediately  implies that $\HF(\hat\psi)$ is well defined and 
$
\HF(g)=\HF(\hat\psi) \cong
\Z_2^{\sum_{x\in\Fix(\psi)}|\Ind(x)|}$
, and  $ \dim\HF(\hat\psi)= \sum_{x\in\Fix(\psi)}|\Ind(x)|$.
On  the last step of the proof, which was fully justified only  recently  by A. Cotton-Clay \cite{cot},  we show that for any symplectomorphism $\phi$ in a pseudo-Anosov mapping class $g$  with nondegenerate fixed points $\HF(\phi)$ is well defined
and $\HF(\phi)=\HF(g)$.

\end{proof}

\begin{rk}
A. Cotton-Clay  found in \cite{cot} a nice  combinatorial  formula
computing $\HF(\hat\psi)$ using train tracs .
\end{rk}

\begin{rk}
 In \cite{eft} Floer homology were calculated for certain class of pseudo-Anosov
maps which are compositions of positive and negative Dehn twists along loops in $M$
forming a tree-pattern. It is interesting to compare the results  above with 
this calculation.

\end{rk}

\subsection{Reducible mapping class}

Recently, in the  paper \cite{cot},  A. Cotton-Clay calculated
 Seidel's symplectic Floer homology for reducible mapping classes.
This result completing previous computations
in the case of arbitrary compositions of Dehn twists along a disjoint collection of curves \cite{S}, in the  case of  periodic mapping classes \cite{G, ff}, as well as
reducible mapping classes in which the map on each component is periodic \cite{G} and  in  the case of certain compositions of Dehn twists, including some  pseudo-Anosov maps \cite{eft}.

In the  case of reducible mapping classes  a energy estimate forbids holomorphic discs from
crossing reducing curves except when a pseudo-Anosov component
meets an identity component ( with no twisting).
Let us introduce  some notation following \cite{cot}.
Recall the notation of $M_{id}$ for the collection of fixed components
as well as the tree types of boundary:  1) $\p_{+}M_{id}$, 
$\p_{-}M_{id}$ denote the collection of  components of $\p M_{id}$
on which we've joined up with a positive(resp. negative)  twist;
2) the collection of components of $\p M_{\id}$ which meet a pseudo-Anosov component will be denoted $\p_pM_{\id}$.
Additionally let $M_1$ be the collection of periodic components and let $M_2$ bethe collection of pseudo-Anosov components with punctures( i.e. before any perturbation) instead of boundary components wherever there is a boundary component that meets a fixed component.
We further subdivide $M_{\id}$. Let $M_a$ be the collection of fixed components
which don't meet any pseudo-Anosov components. Let $M_{b,p}$ be the collection of fixed components which meet one pseudo-Anosov component at a boundary with
$p$ prongs. In this case, we assign the boundary components to
$\p_{+}M_{\id}$ (this is an arbitrary choice). Let $M^o_{b,p}$ be
the collection of the $M_{b,p}$ with each component punctured once.
Let $M_{c,q}$ be the collection of fixed components which meets at least
two pseudo-Anosov components such that the total number of prongs over all
the boundaries is $q$. In this case, we assign at least one boundary
component to $\p_{+}M_{\id}$ and at least one to $\p_{-}M_{\id}$
(and beyond that, it does not matter).

\begin{theorem}( A. Cotton-Clay \cite{cot})
Let $\phi$ be a perturbed standard form map \cite{cot} in a reducible
mapping class $g$ with choices of the signs of components of $\p_{p} M_{\id}$.Then $\HF(\phi)$ is well-defined and 
$$
\HF(g)=\HF(\phi)\cong H_{*}(M_a,\p_{+}{M_{\id}};\Z_2) \oplus
$$
$$
\oplus_{p}(H_{*}(M_b^o,\p_{+}M_{b};\Z_2)\oplus(\Z_2)^{(p-1)|\pi_{0}(M_{b,p})|}) \oplus
$$
$$
\oplus_{q}(H_{*}(M_c,\p_{+}M_{c};\Z_2)\oplus(\Z_2)^{q|\pi_{0}(M_{c,q})|}) \oplus
$$
$$
\oplus\Z_2^{L(\phi|M_1)}\oplus CF_{*}(\phi|M_2),
$$
where $L(\phi|M_1)$ is the Lefschetz number of  $\phi|M_1$,
the $\Z_2^{L(\phi|M_1)}$ summand is all in even degree, the other
two $\Z_2$ summands are all in odd degree, and $CF_{*}(\phi|M_2)$
denotes the chain complex for $\phi$ on the components $M_2$

\end{theorem}

\begin{rk}
The first summand and the $\Z_2^{L(\phi|M_1)}$ are as in R. Gautschi's 
Theorem \ref{th:ga} \cite{G}. The last summand comes from the fact that there
are no differentials in the Floer chain complex on the pseudo-Anosov components
as in the proof of Theorem \ref{th:pa}. The sums  over $p$ and $q$ arise
in the same manner as the first summand.  
\end{rk}

\begin{corollary}
As an application, A. Cotton-Clay gave recently \cite{cot1, cot2} a sharp lower
bound on the number of fixed points of an area-preserving map( with nondegenerate fixed points) in any prescribed mapping class(rel boundary).
\end{corollary}

\subsection{Concluding remarks}

\begin{itemize}

\item  Due to P. Seidel  \cite{S1}   $\dim\HF(\phi)$  is a new symplectic invariant
of a four-dimensional symplectic  manifold  with nonzero first Betti number.
This 4-manifold produced from symplectomorphism $ \phi$ by a surgery construction which is a variation of earlier constructions due to McMullen-Taubes,
Fintushel-Stern and J. Smith.We hope that our  asymptotic invariant  and symplectic zeta function $F_{\phi}(z)$  also
give rise to a new  invariants  of contact 3- manifolds
and symplectic 4-manifolds.

\item { \bf Nielsen Floer homology.}    
Let $\phi: M^{2n}\to M^{2n}$ be a  symplectomorphism of a compact symplectic
manifold $(M^{2n},\omega)$. Suppose that  all fixed points of $\phi$ are  nondegenerate.The symplectomorphism $\phi$ has finite number of Nielsen fixed point
classes $F_1, F_2, ...., F_n$ and there is   no pseudoholomorphic curves( instantons) between fixed
points in the different Nielsen fixed point classes because there is no Nielsen
discs between them. This simple observation strongly suggests to consider for every $i$ a local  Floer complex $CF_{*}(\phi, F_i)$   generated by nondegenerate fixed points in the Nielsen fixed point class $F_i$ and to define for every $i$ a local Nielsen Floer
homology $NFH_{*}(\phi, F_i)$ and then to define  Nielsen Floer homology
$NFH_{*}(\phi)$ of symplectomorphism $\phi$ as direct sum over all Nielsen fixed point classes of local Nielsen Floer homology:
 $$NFH_{*}(\phi)=\oplus_{F_i} NFH_{*}(\phi, F_i).$$
\begin{que}
Under which conditions on $\phi\in Symp(M^{2n},\omega)$ Nielsen Floer homology
are well-defined?
\end{que}
The results discussed in the  present article strongly suggest that a calculation of Nielsen Floer homology and investigation of their properties  should  be important for a higher dimensional generalisation of Poincar\'e-Birkhoff theorem and Arnold conjecture in nondegenerate case.
The author came to the idea of Nielsen Floer homology in 2001 after discussions
with V. Turaev. A lack of a weak version of a monotonicity for  the symplectomorphism $\phi$ prevented the development of this theory at that moment.
Recently, Andrew  Cotton-Clay \cite{cot} defined weak monotonicity for 
a given Nielsen fixed point class $F$ of  symplectomorphism  of surface
and proved that Nielsen Floer homology are well defined in this case.
 The properties of $F$-weak monotone symplectomorphism
of surface  play a crucial role in  his  computation of Seidel's symplectic Floer homology for reducible mapping classes.  

\end{itemize}


\begin{thebibliography}{10}

\bibitem{B} G.D. Birkhoff, An extension of Poincar\'e's last geometric theorem,
Acta mathematica, 47(1925), 297 - 311.

\bibitem{conzeh} C.Conley and E. Zehnder, The Birkhoff-Lewis fixed point theorem
and a conjecture of V.I. Arnold. Inventiones Mathematicae, 73(1983), 33-49.
\bibitem{cot}
A. Cotton-Clay, Symplectic Floer homology of area-preserving surface diffeomorphisms, Preprint, Berkeley, June  2008.

\bibitem{cot1}
A. Cotton-Clay, Symplectic Floer homology of pseudo-Anosov and reducible maps.
A talk on the Northern California Symplectic Geometry Seminar, May 5, 2008.
\bibitem{cot2}
A. Cotton-Clay, A sharp bound on fixed points  of area-preserving surface diffeomorphisms, In preparation.
\bibitem{ds}
S.~Dostoglou and D.~Salamon,  Self dual instantons and holomorphic
  curves, Annals of Math., 139 (1994), 581--640.

\bibitem{eft} E. Eftekhary, Floer cohomology of certain pseudo-Anosov maps,
J. Symplectic Geom.2(2004), no.3, 357-375.

\bibitem{eli} Ya. Eliashberg, Estimates on the number of fixed points of area preserving
transformations. Syktyvkar University Preprint, 1979.


\bibitem{flp}
A.~Fathi, F.~Laudenbach, and V.~Po{\'e}naru, Travaux de Thurston sur
  les surfaces, Ast{\'e}risque, vol. 66--67, Soc. Math. France, 1979.
\bibitem{fv}
A. L. Fel'shtyn,
New zeta function in dynamic.
in Tenth Internat. Conf. on Nonlinear Oscillations,
Varna, Abstracts of Papers, Bulgar. Acad. Sci., 1984, 208
\bibitem{fl}
A.L. Fel'shtyn,
New zeta functions for dynamical systems
   and Nielsen fixed point theory.
in : Lecture Notes in Math. 1346, Springer, 1988, 33-55.


\bibitem{FelshB}
A.~Fel'shtyn, \emph{Dynamical zeta functions, {N}ielsen theory and
  {R}eidemeister torsion}, Mem. Amer. Math. Soc. \textbf{147} (2000),
 no.~699,
  xii+146. MR{2001a:37031}

\bibitem{FelHill} A. L. Fel'shtyn and R. Hill, The Reidemeister zeta
 function
with applications to Nielsen theory and a connection with Reidemeister
 torsion,  {\em K-theory}
{\bf 8} no.4 (1994), 367--393.


\bibitem{ff}
Fel'shtyn A.L. Floer homology, Nielsen theory and symplectic zeta
functions. Proceedings of the Steklov Institute of Mathematics, Moscow, vol. 246,  2004,  pp. 270-282.


\bibitem{fa} Fel'shtyn A.L. New directions in Nielsen-Reidemeister theory.
E-print arxiv: math.GR/0712.2601.


\bibitem{FS}
R. Fintushel, R. Stern, Knots, links and 4-manifolds, Invention. Math. 134(1998), 363-400.
\bibitem{Floer1}
A.~Floer.
Morse theory for Lagrangian intersections.
J. Differential Geom., 28(1988), 513-547.

\bibitem{Floer}
A.~Floer.
 Symplectic fixed points and holomorphic spheres.
 Comm. Math. Phys., 120(2),575--611, 1989.
\bibitem{fukono} K. Fukaya, K. Ono,  Arnold conjecture and Gromov-Witten invariants for general
symplectic manifolds, Topology 38(1999), 933-1048.

 \bibitem{G}
R.~Gautschi.
Floer homology of algebraically finite mapping classes
 J. Symplectic Geom. 1(2003), no.4, 715-765.

\bibitem{Gromov}
M.~Gromov.
Pseudoholomorphic curves in symplectic manifolds.
 Invent. Math.82(1985), 307-347.


\bibitem{hofsal} H. Hofer and D. Salamon, Floer homology and Novikov's ring,
in Floer Memorial Volume, 1995, 483-524, Birkhauser.
\bibitem{iva} N. V. Ivanov, Automorphisms of Teichm\:uller modular
 groups, {\em Lecture Notes in Mathematics} {\bf 1346} Springer,
 Berlin,
  1988.
\bibitem{HS} M. Hutchings, M. Sullivan, The periodic Floer homology of Dehn twist,
Alg. Geom. Topol. 5(2005), 301 - 354.
 
\bibitem{IP} E.-N. Ionel, Th. Parker, Gromov invariants and symplectic maps,
Math. Ann. 314(1999), 127-158.


\bibitem{I}
N. V.~ Ivanov,
Nielsen numbers of maps of surfaces.
Journal  Sov. Math., 26, (1984).
 \bibitem{i1}
N. V. Ivanov,
Entropy and the Nielsen Numbers.
Dokl. Akad. Nauk SSSR 265 (2) (1982), 284-287 (in Russian);
English transl.:
Soviet Math. Dokl. 26 (1982), 63-66.





\bibitem{j}
B.~Jiang, \emph{Lectures on {N}ielsen fixed point theory}, Contemp.
 Math.,
  vol.~14, Amer. Math. Soc., Providence, RI, 1983.
\bibitem{J}
B.~Jiang.
 Fixed point classes from a differentiable viewpoint.
 In  Fixed point theory,  Lecture Notes in
  Math.,vol. 886, 163--170. Springer, 1981.
\bibitem{j1}
B. Jiang,
Estimation of the number of periodic orbits.
Pacific Jour. Math., 172(1996), 151-185.
\bibitem{JG}
B.~Jiang and J.~Guo.
 Fixed points of surface diffeomorphisms.
 Pac. J. Math., 160(1):67--89, 1993.
\bibitem{KH}
A. Katok, B. Hasselblatt, Introduction to the modern theory of dynamical systems,
Cambridge University Press, 1995. 
\bibitem{KM} 
P. Kronheimer, T. Mrowka, Monopoles and Three - Manifolds,
Cambridge University Press, 2007.

\bibitem{liutian} G. Liu and G. Tian, Floer homology and Arnold conjecture,
J.Diff.Geom., 49(1998), 1-74.

\bibitem{MS}
D.~McDuff and D.~A. Salamon.
  Introduction to symplectic topology.
 Oxford Mathematical Monographs. Oxford Science Publications, 1998.
\bibitem{MS1}
D.~McDuff and D.~A. Salamon.
J-holomorphic Curves and Symplectic Topology. AMS Colloquium Publications,
Vol. 52, 2004.

\bibitem{Mo}
J.~Moser.
 On the volume elements on a manifold.
  Trans. Amer. Math. Soc., 120:286--294, 1965.
\bibitem{OS} P. Ozsv\'ath, Z. Szab\'o, Holomorphic disks and topological invariants
for closed 3-manifolds, Annals of Math. 159(2004), 1027-1158.

\bibitem{N}
J. Nielsen. Surface transformations of algebraically finite type. Danske Vid. Salsk. Math.-Phys., 21, 1944. 
\bibitem{pf}
V. B. Pilyugina and A. L. Fel'shtyn,
The Nielsen zeta function.
Funktsional. Anal. i Prilozhen. 19 (4) (1985), 61-67 (in Russian);
English transl.:
Functional Anal. Appl. 19 (1985), 300-305.


\bibitem{S}
P.~Seidel.
 Symplectic Floer homology and the mapping class group.
  Pacific J. Math. 206(2002), no. 1, 219-229.
\bibitem{S1}
P.~Seidel.
Braids and symplectic four-manifolds with abelian fundamental group.
Turkish  J. Math. 26(2002), no.1, 93-100.

\bibitem{R}
Yu. Rudyak , On category weight and its applications, Topology 38(1999), no.1, 37-55.
\bibitem{RO} Yu. Rudyak , J. Oprea, On the Lusternik - Schnirelmann category of symplectic manifolds and the Arnold conjecture,Math.Z.230(1999), no.4, 673-678. 

\bibitem{th}
W. Thurston,
The geometry and topology of 3-manifolds.
Princeton University, 1978.
\bibitem{Th}
W.~P. Thurston.
 On the geometry and dynamics of diffeomorphisms of surfaces.
 Bull. Amer. Math. Soc., 19(2):417--431, 1988.



\end{thebibliography}
\end{document}